\documentclass[11pt]{amsart}
\usepackage{amsmath, amssymb,amscd,verbatim,appendix,lpic,multicol}
\usepackage[mathscr]{eucal}
\usepackage{amscd}
\usepackage{stmaryrd}
\usepackage{enumerate}
\usepackage{amsthm}
\usepackage[margin=1in]{geometry}

\usepackage{enumerate,url,float,tikz,array}

\newtheorem{theorem}{Theorem}
\newtheorem{lemma}[theorem]{Lemma}
\newtheorem{corollary}[theorem]{Corollary}
\newtheorem{proposition}[theorem]{Proposition}

\newtheorem{fact}[theorem]{Fact}

\theoremstyle{definition}
\newtheorem{definition}[theorem]{Definition}
\newtheorem{example}[theorem]{Example}
\newtheorem{remark}[theorem]{Remark}

\newtheorem{question}[theorem]{Question}

\numberwithin{equation}{section}

\newcommand{\cA}{\mathcal{A}}
\newcommand{\cB}{\mathcal{B}}
\newcommand{\cC}{\mathcal{C}}

\newcommand{\cE}{\mathcal{E}}
\newcommand{\cF}{\mathcal{F}}
\newcommand{\cG}{\mathcal{G}}

\newcommand{\cH}{\mathcal{H}}

\newcommand{\cK}{\mathcal{K}}
\newcommand{\cL}{\mathcal{L}}

\newcommand{\cM}{\mathcal{M}}
\newcommand{\cN}{\mathcal{N}}

\newcommand{\cR}{\mathcal{R}}
\newcommand{\cS}{\mathcal{S}}
\newcommand{\cU}{\mathcal{U}}

\def\N{\mathbb N}

\def\Q{\mathbb Q}
\def\R{\mathbb R}

\newcommand{\p}{\oplus}

\newcommand{\Fraisse}{Fra\"{i}ss\'{e}}
\newcommand{\seq}{\subseteq}
\newcommand{\vphi}{\varphi}

\newcommand{\xbar}{\bar{x}}

\def\Aut{\operatorname{Aut}}
\def\Isom{\operatorname{Isom}}
\def\dom{\operatorname{dom}}
\def\im{\operatorname{Im}}

\renewcommand{\Im}{\im}

\AtBeginDocument{%
   \def\MR#1{}
}

\begin{document}

\title[Extending partial isometries]{Extending partial isometries of generalized metric spaces}

\author{
Gabriel Conant \\
University of Notre Dame\\
gconant@nd.edu
}

\date{September 10,2018}

\begin{abstract}
We consider generalized metric spaces taking distances in an arbitrary ordered commutative monoid, and investigate when a class $\cK$ of finite generalized metric spaces satisfies the Hrushovski extension property: for any $A\in\cK$ there is some $B\in\cK$ such that $A$ is a subspace of $B$ and any partial isometry of $A$ extends to a total isometry of $B$. We prove the Hrushovski property for the class of finite generalized metric spaces over a semi-archimedean monoid $\cR$. When $\cR$ is also countable, we use this to show that the isometry group of the Urysohn space over $\cR$ has ample generics. Finally, we prove the Hrushovski property for classes of integer distance metric spaces omitting metric triangles of uniformly bounded odd perimeter. As a corollary, given odd $n\geq 3$, we obtain ample generics for the automorphism group of the universal, existentially closed graph omitting cycles of odd length bounded by $n$.
\end{abstract}



\maketitle

\section{Introduction}\label{sec:intro}
\setcounter{theorem}{0}
\numberwithin{theorem}{section}

A well-known result of Hrushovski \cite{HrEPA} from 1992 says that, given a finite graph $\Gamma_1$, there is a finite graph $\Gamma_2$ such that $\Gamma_1$ is an induced subgraph of $\Gamma_2$ and any partial automorphism of $\Gamma_1$ extends to a total automorphism of $\Gamma_2$. Since then, several authors have investigated the occurrence of similar behavior in other classes of finite structures. For example, Herwig (\cite{HeEPA},\cite{HeEPA2}) proves analogous results for the class of finite $K_n$-free graphs; and Solecki \cite{Sol} proves analogous results for the class of finite metric spaces. 

In \cite{HHLS}, Hodges, Hodkinson, Lascar, and Shelah use Hrushovki's original result, along with an analysis of generic automorphisms, to prove that the automorphism group of the countable random graph has the small index property. In \cite{KeRo}, Kechris and Rosendal generalize this substantially by determining conditions on a \Fraisse\ class $\cK$, which characterize when the \Fraisse\ limit $\cM$ of $\cK$ has ``ample" generic automorphisms. They then show that this implies the small index property for the automorphism group of $\cM$. Translated to a general \Fraisse\ class $\cK$, Hrushovski's extension property for partial automorphisms can make considerable progress toward showing that $\cK$ satisfies the conditions in Kechris and Rosendal's characterization.

In this paper, we consider classes of generalized metric spaces, which take distances in arbitrary ordered commutative monoids, called \emph{distance monoids} (Definition \ref{def:DM}). This is a very robust setting, which naturally encompasses many of the previous examples, and also provides a convenient framework for obtaining new results. In particular, we investigate natural algebraic properties of a distance monoid $\cR$, which imply that the class of finite metric spaces over $\cR$ has the Hrushovski property. As a starting point, we observe in Theorem \ref{thm:arch} that Solecki's work in \cite{Sol} generalizes easily to show the Hrushovski property for $\cK_\cR$ in the case that $\cR$ is archimedean. This provides the base case for proving the same result when $\cR$ is \emph{semi-archimedean} (Definition \ref{def:SA}), which is much less restrictive and includes, for example, ultrametric spaces as a special case. Altogether we conclude that, when $\cR$ is countable and semi-archimedean, the isometry group of the Urysohn space over $\cR$ has ample generics. Finally, we extend our results to classes of generalized metric spaces omitting some fixed class of subspaces. This yields the Hrushovski property for certain classes of metric spaces omitting triangles of odd perimeter. Using this, we obtain ample generics for the automorphism group of the universal, existentially closed graph omitting cycles of uniformly bounded odd length.

We begin with a generalization of the notion of a metric space, in which distances between points come from arbitrary ordered additive structures. 

\begin{definition}\label{def:DM}$~$
\begin{enumerate}[\hspace{0pt}(1)]
\item A structure $\cR=(R,\p,\leq,0)$ is a \textbf{distance monoid} if
\begin{enumerate}[$(i)$]
\item $(R,\leq,0)$ is a linear order with least element $0$;
\item $(R,\p,0)$ is a commutative monoid with identity $0$;
\item for all $r,s,t,u\in R$, if $r\leq s$ and $t\leq u$ then $r\p t\leq s\p u$.
\end{enumerate}
\item Suppose $\cR=(R,\p,\leq,0)$ is a distance monoid. Given a set $A$ and a function $d\colon A\times A\to R$, we call $(A,d)$ an \textbf{$\cR$-metric space} if
\begin{enumerate}[$(i)$]
\item for all $a,b\in A$, $d(a,b)=0$ if and only if $a=b$;
\item for all $a,b\in A$, $d(a,b)=d(b,a)$;
\item for all $a,b,c\in A$, $d(a,c)\leq d(a,b)\p d(b,c)$.
\end{enumerate}
\item Given a distance monoid $\cR$, let $\cK_\cR$ be the class of finite $\cR$-metric spaces.
\end{enumerate}
\end{definition}

If $\cR$ is a \emph{countable} distance monoid, then $\cK_\cR$ is a countable class modulo isometry. In this case, it is straightforward to verify that $\cK_\cR$ is a \Fraisse\ class, using the natural notion of free amalgamation of $\cR$-metric spaces. In particular, given finite $\cR$-metric spaces $A=(A,d_A)$ and $B=(B,d_B)$, with $A\cap B\neq\emptyset$, define the \emph{free amalgamation of $A$ and $B$}, denoted $A\otimes B$, to be the $\cR$-metric space $C=(C,d_C)$, where $C=A\cup B$ and, for all $x,y\in C$,
$$
d_C(x,y)=\begin{cases}
d_A(x,y) & \text{if $x,y\in A$}\\
d_B(x,y) & \text{if $x,y\in B$}\\
\min_{z\in A\cap B}(d_A(x,z)\p d_B(z,y)) & \text{if $x\in A\backslash B$ and $y\in B\backslash A$.}
\end{cases}
$$ 
The operation $A\otimes B$ can be extended to the case where $A\cap B=\emptyset$ by setting, for $x\in A$ and $y\in B$, $d_C(x,y)$ to be the largest distance occurring in $A$ or $B$. Together, this gives the amalgamation property and the (dis)joint embedding property for $\cK_\cR$ (the hereditary property is obvious).

\begin{definition}
Given a countable distance monoid $\cR$, define the \textbf{$\cR$-Urysohn space}, denoted $\cU_\cR$, to be the countable \Fraisse\ limit of $\cK_\cR$.
\end{definition}

By standard results in classical \Fraisse\ theory, we have the following fact.

\begin{fact}
Suppose $\cR$ is a countable distance monoid. Then $\cU_\cR$ is the unique (up to isometry) countable, ultrahomogeneous $\cR$-metric space, which is universal for the class $\cK_\cR$ of finite $\cR$-metric spaces. 
\end{fact}

\begin{example}\label{example}
$~$
\begin{enumerate}[\hspace{0pt}(1)]
\item Classical metric spaces are generalized metric spaces over the standard distance monoid $\R^{\geq0}=(\R^{\geq0},+,\leq,0)$. A countable example is  the submonoid $\Q^{\geq0}=(\Q^{\geq0},+,\leq,0)$ of rational numbers, which yields the \emph{rational Urysohn space} $\cU_{\Q^{\geq0}}$. The completion of $\cU_{\Q^{\geq0}}$ is known as the \emph{Urysohn space}, and is the unique separable ultrahomogeneous metric space, which is universal for the class of separable metric spaces. 
\item Ultrametric spaces are generalized metric spaces over the distance monoid $(\R^{\geq0},\max,\leq,0)$. In general, if $(R,\leq,0)$ is a linear order with least element $0$, then $\cR=(R,\max,\leq,0)$ is a distance monoid. When $R$ is countable, we have the associated \emph{ultrametric Urysohn space} $\cU_\cR$. Any distance monoid arising in this way is called \emph{ultrametric}. 

\item We will often focus on metric spaces taking only integer distances. Therefore, given an integer $n>0$, we define the distance monoid $\cS_n=(\{0,1,\ldots,n\},+_n,\leq,0)$, where $r+_n s:=\min\{r+s,n\}$. We also define $\cN=(\N,+,\leq,0)$.

\item The previous example is a special case of the following. Fix $S\seq\R^{\geq0}$ such that $0\in S$ and $S$ is closed under the binary operation $r+_S s=\sup\{x\in S:x\leq r+s\}$. When $+_S$ is associative, we have a distance monoid $\cS=(S,+_S,\leq,0)$. Associativity of $+_S$ is equivalent to what is known as the \emph{four-values condition} for $S$, which describes the ability to amalgamate metric spaces with distances in $S$. This condition was defined by Delhomm\'{e}, Laflamme, Pouzet, and Sauer in \cite{DLPS}, where these authors study metric spaces with restricted (real) distances. 
\end{enumerate}
\end{example}

Part (3) of the previous example provides a convenient setting in which to work with the natural metrics on graphs.

\begin{definition}
Suppose $(\Gamma,E)$ is a graph.
\begin{enumerate}[\hspace{0pt}(1)]
\item If $(\Gamma,E)$ is connected, then we define the \textbf{path-metric}, $d\colon\Gamma\times\Gamma\to \N$, by setting $d(x,y)$ to be the number of edges in the shortest path from $x$ to $y$. In this case, $(\Gamma,d)$ is an $\cN$-metric space. If $(\Gamma,d)$ has bounded diameter then it is an $\cS_n$-metric space for large enough $n$.
\item For $n>0$, we define the \textbf{path-metric truncated at $n$}, $d\colon \Gamma\times \Gamma\to \{0,1,\ldots,n\}$, by setting $d(x,y)$ to be the minimum of $n$ and the number of edges in the shortest path from $x$ to $y$ (if one exists). Then $(\Gamma,d)$ is an $\cS_n$-metric space. 
\end{enumerate}
\end{definition}

Our interest lies in questions around extending partial isometries of generalized metric spaces, and the resulting consequences for the isometry groups of universal generalized metric spaces (e.g. $\cU_\cR$ for a distance monoid $\cR$). We begin this investigation with the basic definitions. 

\begin{definition}
Fix a distance monoid $\cR$.
\begin{enumerate}[\hspace{0pt}(1)]
\item Given an $\cR$-metric space $A$, a \textbf{partial isometry of $A$} is an isometry $\vphi\colon A_1\to A_2$, where $A_1$ and $A_2$ are subspaces of $A$.
\item Suppose $\cK$ is a class of $\cR$-metric spaces and $A\in\cK$. Then $A$ has the \textbf{extension property in $\cK$} if there is $B\in\cK$ such that $A$ is a subspace of $B$ and any partial isometry of $A$ extends to a total isometry of $B$.
\item A class $\cK$ of finite $\cR$-metric spaces has the \textbf{Hrushovski property} if every element of $\cK$ has the extension property in $\cK$.
\end{enumerate}
\end{definition}

In the literature (e.g. \cite{HeLa}), the Hrushovski property is sometimes also referred to as EPPA: the \emph{extension property for partial automorphisms}.

For graphs, the path-metric truncated at $2$ associates $\cK_{\cS_2}$ with the class of finite graphs, and $\cU_{\cS_2}$ with the countable random graph. Viewed this way, Hrushovski's original result in \cite{HrEPA} is that $\cK_{\cS_2}$ has the Hrushovski property. In \cite{Sol}, Solecki shows that $\cK _{\R^{\geq 0}}$, the class of all finite metric spaces, also has the Hrushovski property. More precisely, it is shown that if $(G,+,0)$ is a subgroup of $(\R,+,0)$, then $\cK_\cG$ has the Hrushovski property, where $\cG$ is the distance monoid $(G^{\geq0},+,\leq,0)$. The only significant obstacle in extending Solecki's result to arbitrary distance monoids is that $\R^{\geq 0}$ is archimedean.

\begin{definition}
A distance monoid $\cR=(R,\p,\leq,0)$ is \textbf{archimedean} if, for all $r,s\in R^{>0}$, there exists some integer $n>0$ such that $s\leq nr$.
\end{definition}

Solecki's work in \cite{Sol} easily generalizes to show the Hrushovski property for $\cK_{\cR}$  for any archimedean distance monoid $\cR$ (see Theorem \ref{thm:arch}). 
However, many interesting examples of generalized metric spaces are obtained using non-archimedean monoids, which leads to the notion of a \emph{semi-archimedean} distance monoid, defined below. To motivate the definition, note that on any distance monoid $\cR$, we have the relation of \emph{archimedean equivalence}: $r\sim s$ if there is some $n>0$ such that $s\leq nr$ and $r\leq ns$). The $\sim$-classes of $\cR$ are called \emph{archimedean classes}, and the \emph{trivial} class is $\{0\}$.  In this terminology, a nontrivial distance monoid $\cR$ is archimedean if and only if it has exactly one nontrivial archimedean class. On the other hand, in an ultrametric monoid $\cR=(R,\max,\leq,0)$, archimedean equivalence coincides with equality, and so $\cR$ is severely non-archimedean. Archimedean and ultrametric distance monoids are extreme cases of the following more general notion.

\begin{definition}\label{def:SA}
A distance monoid $\cR=(R,\p,\leq,0)$ is \textbf{semi-archimedean} if, for all $r,s\in R^{>0}$, if $nr<s$ for all $n>0$ then $r\p s=s$.
\end{definition}

In other words, a semi-archimedean distance monoid can have multiple nontrivial archimedean classes, but addition between elements of distinct archimedean classes is as trivial as possible. Our first main result is the following theorem, which is proved in Section \ref{sec:SA}.

\begin{theorem}\label{thm:SA}
If $\cR$ is a semi-archimedean distance monoid, then $\cK_\cR$ has the Hrushovski property.
\end{theorem}

The proof of this theorem is by induction on the number of nontrivial archimedean classes of $\cR$. The base case ($\cR$ archimedean) is proved by generalizing \cite{Sol}, as discussed above. The inductive step uses a ``composition operation" from \cite{SokRQ} on classes of relational structures, as well as the fact that, for classes with joint embedding, this operation preserves the Hrushovski property (shown by Pawliuk and Soki\'{c} \cite{PaSo}).  The class of semi-archimedean distance monoids is essentially the largest class for which this proof strategy works (see Propositions \ref{prop:glue} and \ref{prop:SAstructure}). Therefore, the question of extending partial isometries of $\cR$-metric spaces, where $\cR$ is a general distance monoid, remains open and seems to demand a new method of proof.

\begin{question}
Let $\cR$ be an arbitrary distance monoid. Does $\cK_\cR$ have the Hrushovski property?
\end{question}

Up to isomorphism, the smallest non-semi-archimedean distance monoid is $\cS=(S,+_S,\leq,0)$, where $S=\{0,1,3,4\}$ (as defined in Example \ref{example}$(4)$). Thus this example gives a good test case for proving the Hrushovski property for larger classes of distance monoids. 

In addition to being a combinatorially interesting property of finite structures, the Hrushovski property can also have significant consequences for automorphism groups of countable structures. Recall that if $\cM$ is a countable structure, then $G=\Aut(\cM)$ is a separable topological group under the pointwise convergence topology. Quoting \cite{KeRo}, we say $G$ has \emph{ample generics} if, for all $n>0$, the conjugation action of $G$ on $G^n$ has a comeager orbit. If $G$ has ample generics then any subgroup of index less than $2^{\aleph_0}$ is open (and therefore has countable index). It then follows that any homomorphism from $G$ to a separable topological group is continuous. See \cite{KeRo} for details.

 In \cite{KeRo}, Kechris and Rosendal characterize the existence of ample generics for $\Aut(\cM)$, in the case that $\cM$ is the \Fraisse\ limit of a \Fraisse\ class $\cK$. This characterization is given in terms of the \emph{joint embedding property} and \emph{weak amalgamation property}, which are defined for the class $\cK^{p,n}$ consisting of tuples $(A,\vphi_1,\ldots,\vphi_n)$ where $A\in\cK$ and each $\vphi_i$ is a partial automorphism of $A$. In the setting of metric spaces, it is shown in \cite[Section 6.2]{KeRo} that the Hrushovski property for $\cK_{\Q^{\geq0}}$ implies the weak amalgamation property for $\cK_{\Q^{\geq0}}^{p,n}$. Together with joint embedding for $\cK_{\Q^{\geq0}}^{p,n}$ (also shown in \cite{KeRo}), one obtains Solecki's result that $\Isom(\cU_{\Q^{\geq0}})$ has ample generics. For any distance monoid $\cR$, the proofs of joint embedding for $\cK_{\cR}^{p,n}$, and weak amalgamation for $\cK_{\cR}^{p,n}$ as a consequence of the Hrushovski property for $\cK_{\cR}$, go through exactly as in \cite{KeRo} for $\cK_{\Q^{\geq 0}}$. Altogether, we have the following conclusion.

\begin{fact}\label{fact:ample}
Suppose $\cR$ is a countable distance monoid and $\cK_\cR$ has the Hrushovski property. Then $\Isom(\cU_\cR)$ has ample generics. 
\end{fact}

Applying Theorem \ref{thm:SA}, we then have:

\begin{corollary}
If $\cR$ is a countable semi-archimedean distance monoid, then $\Isom(\cU_\cR)$ has ample generics.
\end{corollary}

We should mention that recent work of Malicki \cite{MalUlt} shows ample generics for the isometry groups of Polish \emph{ultrametric} Urysohn spaces. 

Finally, in Section \ref{sec:O}, we consider the Hrushovski property for classes of generalized metric spaces omitting some fixed class of subspaces. In particular, given a distance monoid $\cR$ and a class $\cF$ of $\cR$-metric spaces, we say an $\cR$-metric space $A$ is \emph{$\cF$-free} if  no subspace of $A$ is isometric to an element of $\cF$. A canonical example is the class of finite $K_n$-free graphs for a fixed $n\geq 3$ (considered as $\cS_2$-metric spaces), and its \Fraisse\ limit which we denote $\cH_n$ (often called a \emph{Henson graph}). Our focus will be the next example.

\begin{example}\label{ex:odd}
Fix an odd integer $n\geq 3$. Let $\cK^c_n$ denote the class of finite graphs which omit all cycles of odd length bounded by $n$ (i.e. contain no such cycle as a subgraph). By work of Komj\'{a}th, Mekler, and Pach (\cite{Kom}, \cite{KMP}), there is a countable, existentially closed graph $\Gamma^c_n$, which omits cycles of odd length bounded by $n$ and contains every graph in $\cK^c_n$ as an induced subgraph. We now translate this situation to the setting of generalized metric spaces. Let $\cF^c_n$ denote the class of \emph{metric triangles of odd perimeter bounded by $n$} (i.e. $3$-point metric spaces with integer distances, in which the sum of the three nonzero distances is odd and bounded by $n$). Define $n_*=\frac{n+1}{2}$ and, for any graph $\Gamma\in\cK^c_n$, let $d_\Gamma$ be the path-metric on $\Gamma$ truncated at $n_*$. Under this translation, $\cK^c_n$ is precisely the class of finite $\cF^c_n$-free $\cS_{n_*}$-metric spaces. Since $\Gamma^c_n$ is existentially closed, it has diameter $n_*$ with respect to the path-metric, and so the path-metric agrees with its truncation at $n_*$. Thus, as an $\cS_{n_*}$-metric space, $\Gamma^c_n$ is $\cF^c_n$-free and isometrically contains  every finite $\cF^c_n$-free $\cS_{n_*}$-metric space. 
\end{example}

The main theorem of Section \ref{sec:O} will imply the Hrushovski property for the classes of generalized metric spaces discussed above. For $K_n$-free graphs, this result is due to Herwig (\cite{HeEPA}, \cite{HeEPA2}), and there is no substantive difference between viewing these structures as graphs versus $\cS_2$-metric spaces. On the other hand, the Hrushovski property for metric spaces omitting metric triangles of odd perimeter has not been previously shown. Moreover, we will use this result to obtain ample generics for the automorphism group of the graph $\Gamma_n^c$. This is interesting because $\Gamma^c_n$ is not homogeneous as a graph and, as a class of graphs, $\cK^c_n$ is not a \Fraisse\ class and does not have the Hrushovski property. However, as a class of $\cS_{n_*}$-metric spaces, $\cK^c_n$ is a \Fraisse\ class. In particular, if $A$ and $B$ are $\cF^c_n$-free $\cS_{n_*}$-metric spaces, then the free amalgamation $A\otimes B$ is still $\cF^c_n$-free. It is also a somewhat folkloric exercise to show that the \Fraisse\ limit of $\cK^c_n$ (as a class of $\cS_{n_*}$-metric spaces) is isometric to $\Gamma^c_n$ (under the path-metric), and so $\Gamma^c_n$ is a \emph{metrically homogeneous graph}.  In fact, along with the Henson graphs, $\Gamma^c_n$ is one of the primary ingredients in an ongoing project, initiated by Cherlin, which aims to classify all metrically homogeneous graphs (see \cite{Chcat},\cite{Chercat}).

To obtain the results described above, in Section \ref{sec:O} we will prove a metric space analog of the Herwig-Lascar Theorem (see Fact \ref{fact:HL} and Theorem \ref{thm:O}), which is the main tool in Solecki's work. Using this, we prove:  

\begin{theorem}\label{thm:oddT}
Fix an odd integer $n\geq 3$.
\begin{enumerate}[$(a)$]
\item Let $n_*=\frac{n+1}{2}$. Then the class $\cK^c_n$ of $\cS_{n_*}$-metric spaces, which omit metric triangles of odd perimeter bounded by $n$, has the Hrushovski property.
\item Let $\Gamma^c_n$ be the countable, universal, existentially closed graph omitting cycles of odd length bounded by $n$. Then $\Aut(\Gamma_n)$ has ample generics.
\end{enumerate}
\end{theorem}

Returning to the general picture, this is only partial progress. In particular, the restrictions in Section \ref{sec:O} placed on the omitted family $\cF$ are quite strong. Moreover, we work only with archimedean distance monoids.

\begin{question}
Let $\cR$ be an arbitrary distance monoid and let $\cF$ be a finite class of finite $\cR$-metric spaces. Suppose the class $\cK$ of $\cF$-free $\cR$-metric spaces satisfies the amalgamation properties of a \Fraisse\ class (but perhaps is uncountable). Under what conditions can we conclude that $\cK$ has the Hrushovski property?
\end{question}

\section{Proof of Theorem \ref{thm:SA}}\label{sec:SA}

In this section, we prove that $\cK_{\cR}$ has the Hrushovski property for any semi-archimedean distance monoid $\cR$. The proof will be by induction on the number of nontrivial archimedean classes in a semi-archimedean distance monoid with finitely many classes. The base case  is precisely when $\cR$ is archimedean, and so we deal with this case right away.

\begin{theorem}\label{thm:arch}
If $\cR$ is an archimedean distance monoid then $\cK_{\cR}$ has the Hrushovski property.
\end{theorem}

This result is a mild generalization of the case $\cR=\R^{\geq0}$, due to Solecki \cite{Sol}, and is also a special case of Theorem \ref{thm:O} below (see Remark \ref{rem:arch}). Toward the induction step of Theorem \ref{thm:SA}, we make the following definition. 

\begin{definition}\label{def:bracket}
Let $\cR_1=(R_1,\p_1,\leq_1,0)$ and $\cR_2=(R_2,\p_2,\leq_2,0)$ be distance monoids. Define a distance monoid $\llbracket \cR_1,\cR_2\rrbracket=(R,\p,\leq,0)$ as follows:
\begin{enumerate}[$(i)$]
\item $R=\{0\}\cup \{(r,1):r\in R_1^{>0}\}\cup\{(r,2):r\in R_2^{>0}\}$;
\item given $i\in\{1,2\}$ and $r,s\in R_i^{>0}$, $(r,i)\p (s,i)=r\p_i s$ and $(r,i)\leq (s,i)$ if and only if $r\leq _i s$.
\item given $r\in R_1^{>0}$ and $s\in R_2^{>0}$, $(r,1)\leq (s,2)$ and $(r,1)\p (s,2)=(s,2)$.
\end{enumerate}
\end{definition}

Informally, $\llbracket \cR_1,\cR_2\rrbracket$ is obtained by taking the disjoint union of $R_1$ and $R_2$ over $\{0\}$, extending the linear order so that $R_1<R_2\backslash\{0\}$, and extending the monoid operation so that $r\p s=s$ for $r\in R_1$ and $s\in R_2$. The next result connects this notion  to work of Soki\'{c} \cite{SokRQ} and Pawliuk-Soki\'{c} \cite{PaSo} on ``composition classes" of relational structures.

\begin{proposition}\label{prop:glue}
Fix distance monoids $\cR_1$ and $\cR_2$.
\begin{enumerate}[$(a)$]
\item $\llbracket \cR_1,\cR_2\rrbracket$ is a distance monoid.
\item Suppose $A=(A,d)$ is an $\llbracket \cR_1,\cR_2\rrbracket$-metric space. Define $\sim$ on $A$ such that, given $x,y\in A$, $x\sim y$ if and only if $d(x,y)\in \cR_1$. Then $\sim$ is an equivalence relation on $A$. If $X\seq A$ is a single $\sim$-class, then $(X,d)$ is an $\cR_1$-metric space. Moreover, there is a well-defined $\cR_2$-metric $d_*$ on $A_*=A/\!\!\sim$  such that, for $\sim$-inequivalent $x,y\in A$, $d_*([x],[y])=d(x,y)$.
\item Let $A=(A,d_A)$ be an $\cR_2$-metric space and, for $x\in A$, let $B_x$ be an $\cR_1$-metric space. Define $C=(C,d_C)$ to be the disjoint amalgamation of $\{B_x:x\in A\}$ where, for distinct $x,y\in A$ and $u\in B_x$, $v\in B_y$, $d_C(u,v)=d_A(x,y)$. Then $C$ is an $\llbracket \cR_1,\cR_2\rrbracket$-metric space.
\item $\cK_{\llbracket \cR_1,\cR_2\rrbracket}$ has the Hrushovksi property if and only if $\cK_{\cR_1}$ and $\cK_{\cR_2}$ have the Hrushovski property. 
\end{enumerate}
\end{proposition}
\begin{proof}
Parts $(a)$, $(b)$, and $(c)$ are straightforward and left to the reader. Parts $(b)$ and $(c)$ are precisely what is needed to verify that $\cK_{\llbracket \cR_1,\cR_2\rrbracket}=\cK_{\cR_2}[\cK_{\cR_1}]$ where, given classes $\cK$ and $\cL$ of relational structures, the composition class $\cK[\cL]$ is defined in \cite{PaSo} (and also in \cite{SokRQ} using the notation $\cL\mathcal{E}\cK$). Since $\cK_\cR$ satisfies joint embedding for any distance monoid $\cR$, part $(d)$ follows as a special case of \cite[Proposition 6.3]{PaSo}.
\end{proof}

\begin{remark}
The right-to-left direction of part $(d)$ (which is the only direction needed for the proof of Theorem \ref{thm:SA}) was shown in the author's thesis \cite[Corollary 4.4.7]{Cothesis} (independently of \cite{PaSo}), and a full proof was given in an earlier draft of this article. We thank the referee for bringing our attention to the recent work of Pawliuk and Soki\'{c} \cite{PaSo}, which provides a general framework for the underlying combinatorics of this result.
\end{remark}

Note that the bracketing operation in Definition \ref{def:bracket} is associative, i.e., $\llbracket \llbracket \cR_1,\cR_2\rrbracket,\cR_3\rrbracket$ is isomorphic to $\llbracket \cR_1,\llbracket \cR_2,\cR_3\rrbracket\rrbracket$. So we have a well-defined notation $\llbracket \cR_1,\cR_2,\cR_3\rrbracket$. This yields the following structural description of semi-archimedean distance monoids.

\begin{proposition}\label{prop:SAstructure}
Fix $n>0$. A distance monoid $\cR$ is semi-archimedean with $n$ nontrivial archimedean classes if and only if there are nontrivial archimedean distance monoids $\cR_1,\ldots,\cR_n$ such that $\cR\cong\llbracket \cR_1,\ldots,\cR_n\rrbracket$.
\end{proposition}
\begin{proof}
It is straightforward to check that if $\cR_1,\ldots,\cR_n$ are nontrivial and archimedean, then $\llbracket \cR_1,\ldots,\cR_n\rrbracket$ is semi-archimedean and has $n$ nontrivial archimedean classes. These properties are invariant under isomorphism.

For the other direction, suppose $\cR$ is semi-archimedean with $n$ nontrivial archimedean classes. Let $\sim$ denote archimedean equivalence in $\cR$ (i.e. $r\sim s$ if and only if there is $n>0$ such that $r\leq ns$ and $s\leq nr$). Given $r,s,t\in R$, if $r\leq s\leq t$ and $r\sim t$, then $r\sim s$. So we have $R=R_0\cup\ldots\cup R_n$, where each $R_i$ is an archimedean class and $R_i<R_j$ for all $i<j$. Note $R_0=\{0\}$ is the trivial class. Next, given $r,s\in R$, if $r\sim s$ then $r\sim r\p s$, and so each $R_i$ is closed under $\p$. So, for each $1\leq i\leq n$, we have an archimedean distance monoid $\cR_i=(R_i\cup\{0\},\p,\leq,0)$. Finally, if $i<j$, $r\in R_i$, and $s\in R_j$, then $r\p s=s$ since $\cR$ is semi-archimedean. Altogether, $\cR\cong\llbracket\cR_1,\ldots,\cR_n\rrbracket$. 
\end{proof}

The reader may note that Definition \ref{def:bracket} can be altered to yield a similar structural result for semi-archimedean distance monoids with infinitely  many archimedean classes (but this will not be needed for our results).  For example, any ultrametric distance monoid $\cR$ can be obtained in this fashion, using (possibly infinitely many) copies of $\cS_1=(\{0,1\},+_1,\leq,0)$. 

The  following is the final observation needed for Theorem \ref{thm:SA}.

\begin{proposition}\label{prop:ctble}
Let $\cR$ be a distance monoid and suppose $\cK_{\cR'}$ has the Hrushovski property for any finitely generated submonoid $\cR'$ of $\cR$. Then $\cK_{\cR}$ has the Hrushovski property.
\end{proposition}
\begin{proof}
For any finite $\cR$-metric space $A$, if $\cR'$ is the submonoid of $\cR$ generated by the distances in $A$, then $A$ is an $\cR'$-metric space and $\cK_{\cR'}\seq\cK_\cR$.
\end{proof}

We now prove the Hrushovski property for $\cK_\cR$ when $\cR$ is semi-archimedean.

\begin{proof}[Proof of Theorem \ref{thm:SA}]
Fix a semi-archimedean distance monoid $\cR$. Since archimedean equivalence is preserved in submonoids, it follows that any submonoid of $\cR$ is again semi-archimedean. By Proposition \ref{prop:ctble}, we may assume $\cR$ is finitely generated, and thus has only finitely many archimedean classes. By Proposition \ref{prop:SAstructure}, we may assume $\cR=\llbracket\cR_1,\ldots,\cR_n\rrbracket$ for some nontrivial archimedean distance monoids $\cR_1,\ldots,\cR_n$. Now the Hrushovski property for $\cK_\cR$ follows by induction, Theorem \ref{thm:SA}, and Proposition \ref{prop:glue}$(d)$ (in the notation of \cite{PaSo}, $\cK_\cR=\cK_{\cR_n}[\ldots[\cK_{\cR_2}[\cK_{\cR_1}]]\ldots]$).
\end{proof}

\section{Metric Spaces with Omitted Subspaces}\label{sec:O}

In this section, we obtain the Hrushovski property for certain classes of generalized metric spaces, which omit some fixed finite class of subspaces. First, we recall the Herwig-Lascar Theorem. Suppose $\cL$ is a finite relational language, and $M$ and $N$ are $\cL$-structures. We write $M\seq N$ to denote that $M$ is a substructure of $N$.  A function $f\colon M\to N$ is a \emph{weak homomorphism} if, given $n$-ary $R\in\cL$ and $\xbar\in M^n$, if $M\models R(\xbar)$ then $N\models R(f(\xbar))$. Given a class $\cF$ of $\cL$-structures, we say $M$ is \emph{$\cF$-free} if there is no weak homomorphism from an element of $\cF$ to $M$.

\begin{fact}[Herwig-Lascar Theorem \cite{HeLa}]\label{fact:HL}
Fix a finite relational language $\cL$ and a finite class $\cF$ of finite $\cL$-structures. Let $\cK$ be the class of $\cF$-free $\cL$-structures. Then, for any finite $A\in\cK$, if there is some $M\in\cK$ such that $A\seq M$ and every partial automorphism of $A$ extends to an automorphism of $M$, then there is a finite $B\in\cK$ such that $A\seq B$ and every partial automorphism of $A$ extends to an automorphism of $B$. 
\end{fact}  

Toward applying this result in the setting of generalized metric spaces, we set the following definitions and notation.

\begin{definition}\label{def:bad}
Fix a distance monoid $\cR$ and a subset $S\seq R^{>0}$. 
\begin{enumerate}[\hspace{0pt}(1)]
\item Define $\cL_S=\{d_r(x,y):r\in S\}$, where each $d_r(x,y)$ is a binary relation. 
\item Given $r\in S$, let $d(x,y)=r$ denote the relation $d_r(x,y)\wedge d_r(y,x)$.
\item Define $\Sigma_\cR^S=\{(r_0,\ldots,r_n)\in S^n:n>0,~r_0>r_1\p\ldots\p r_n\}$.
\item Given $\sigma=(r_0,\ldots,r_n)\in \Sigma_\cR^S$, define an $\cL_S$-structure $P_\sigma=\{x_0,\ldots,x_n\}$ such that $P_\sigma\models d(x_0,x_n)=r_0$ and $P_\sigma\models d(x_{i-1},x_i)=r_i$ for $1\leq i\leq n$. 
\item Define $\cE^S_\cR=\{P_\sigma:\sigma\in\Sigma^S_\cR\}$.
\end{enumerate}
\end{definition}

\begin{remark}\label{rem:metricL}
Let $\cR$ be a distance monoid, and fix $S\seq R^{>0}$. We can consider $\cR$-metric spaces as $\cL_S$-structures where $d_r(x,y)$ is interpreted as ``distance equal to $r$". By the triangle inequality, any $\cR$-metric space is $\cE^S_\cR$-free as an $\cL_S$-structure. Note also that any weak $\cL_S$-homomorphism between $\cR$-metric spaces, with distances in $S$, must be injective. Therefore, if $\cF$ is a class of $\cR$-metric spaces with distances in $S$, and $A$ is an $\cR$-metric space with distances in $S$, then $A$ is $\cF$-free as an $\cR$-metric space if and only if it is $\cF$-free as an $\cL_S$-structure. 
\end{remark} 

In the setting of the previous remark, while any $\cR$-metric space is an $\cE^S_\cR$-free $\cL_S$-structure, the converse need not hold. Indeed, the main work in Solecki's proof of the Hrushovski property for $\cK_{\R^{\geq0}}$ involves an application of the minimal-path metric to extract a metric space from an $\cE^S_{\R^{\geq0}}$-free $\cL_S$-structure. Therefore, in order to add a new forbidden metric space $A$, we must also omit $\cL_S$-structures which can potentially contain a copy of $A$ after applying the minimal-path metric. This motivates the following definitions.

\begin{definition}\label{def:ext}
Fix a distance monoid $\cR=(R,\p,\leq,0)$. Given a set $X$ and a partial function $\delta\colon X\times X\to R$, we call $(X,\delta)$ a \textbf{partial $\cR$-semimetric space} if the following conditions hold for any $x,y\in X$:
\begin{enumerate}[$(i)$]
\item $(x,x)\in\dom(\delta)$ and $\delta(x,x)=0$;
\item if $(x,y)\in\dom(\delta)$ then $(y,x)\in\dom(\delta)$ and $\delta(x,y)=\delta(y,x)$. 
\end{enumerate}
\end{definition}

The next definition is rather technical, and so we first give informal motivation. In particular, given an $\cR$-metric space $(A,d_A)$, we want to consider partial $\cR$-semimetric spaces which contain an isometric copy of $A$ after applying the minimal-path metric. The most direct way to obtain such a semimetric space is to start with $A$ and, for certain pairs $(x,y)\in A\times A$, remove the distance $d_A(x,y)$ and add a path of new points from $x$ to $y$, with defined distances along each step of the path summing to $d_A(x,y)$. A semimetric space obtained in this way will be called a ``path substitution" of $A$. We now give the formal definition (see also Figure \ref{pathsub} below).

\begin{definition}
Fix a distance monoid $\cR=(R,\p,\leq,0)$. Suppose $A=(A,d)$ is an $\cR$-metric space and $(X,\delta)$ is a partial $\cR$-semimetric space, with $A\seq X$. Then $(X,\delta)$ is a \textbf{path substitution of $A$} if  $d(a,b)=\delta(a,b)$ for all $(a,b)\in\dom(\delta)\cap(A\times A)$ and, for all $(x,y)\in A_\delta:=(A\times A)\backslash\dom(\delta)$, there is a sequence $\pi_{x,y}=(x_0,\ldots,x_n)$ of elements of $X$ such that:
\begin{enumerate}[$(i)$]
\item $X=A\cup\bigcup_{(x,y)\in A_\delta}\pi_{x,y}$;
\item for all $(x,y)\in A_\delta$, if $\pi_{x,y}=(x_0,\ldots,x_n)$ then $x_0=x$, $x_n=y$, $(x_i,x_{i+1})\in\dom(\delta)$ for all $0\leq i<n$, and
$$
\delta(x_0,x_1)\p \delta(x_1,x_2)\p\ldots\p \delta(x_{n-1},x_n)=d(x,y)
$$
(we call $\pi_{x,y}$ a \textbf{distance path for $(x,y)$ in $(X,\delta)$}).
\end{enumerate}
If $\Im(\delta)\seq S\cup\{0\}$, for some $S\seq R^{>0}$, then we say $(X,\delta)$ is a \textbf{path substitution of $A$ over $S$}.
\end{definition}

In order to omit a generalized metric space $A$ using Fact \ref{fact:HL}, it will be necessary to also omit all path substitutions of $A$. Therefore, we must determine conditions under which $A$ has only finitely many path substitutions.

\begin{definition}
Let $\cR$ be a distance monoid. An $\cR$-metric space $(A,d)$ is \textbf{$\cR$-dominated} if there is $r\in R$ such that $d(a,b)<r$ for all $a,b\in A$.
\end{definition}

\begin{lemma}\label{lem:ext}
Suppose $\cR$ is an archimedean distance monoid and $S\seq R^{>0}$ is finite and nonempty. Let $A=(A,d)$ be a finite $\cR$-dominated $\cR$-metric space. Then there are only finitely many path substitutions of $A$ over $S$.
\end{lemma} 
\begin{proof}
Since $S$ is finite, it suffices to produce a uniform bound on $|X|$ where $(X,\delta)$ is a path substitution of $A$ over $S$. Fix $r\in R$ such that $d(a,b)<r$ for all $a,b\in A$. Set $s=\min S$. Since $\cR$ is archimedean, there is some $k>0$ such that $r\leq ks$. We show that if $(X,\delta)$ is a path substitution of $A$ over $S$, then $|X|< |A|+k|A|^2$. To see this, fix a path substitution $(X,\delta)$ of $A$ over $S$, and let $A_\delta=(A\times A)\backslash\dom(\delta)$. We have $X=A\cup\bigcup_{(x,y)\in A_\delta}\pi_{x,y}$, where $\pi_{x,y}$ is a distance path for $(x,y)$ in $(X,\delta)$. Since $|A_\delta|\leq|A|^2$, it suffices to show that, for all $(x,y)\in A_\delta$, if $\pi_{x,y}=(x_0,\ldots,x_n)$ then $n<k$. For this, note that $ns\leq \delta(x_0,x_1)\p\ldots\p\delta(x_{n-1},x_n)=d(x,y)<r\leq ks.$
\end{proof}

If $\cF$ is a class of $\cR$-metric spaces, then there is no guarantee that an $\cF$-free $\cR$-metric space will also omit path substitutions of the spaces in $\cF$. Therefore, we must artificially impose this condition.

\begin{definition}
Let $\cR$ be a distance monoid and $\cF$ a class of $\cR$-metric spaces. Fix $S\seq R^{>0}$.
\begin{enumerate}[\hspace{0pt}(1)]
\item Set $\cF^*_S=\{(X,\delta):(X,\delta)\text{ is a path substitution of some $A\in\cF$ over $S$}\}$.
\item $\cF$ is \textbf{$\cR$-closed under path substitutions over $S$} if any $\cF$-free $\cR$-metric space is $\cF_S^*$-free (as an $\cL_S$-structure via Remark \ref{rem:metricL}).
\end{enumerate}
When the ambient distance monoid  $\cR$ is clear from context, we will just say $\cF$ is \textbf{closed under path substitutions over $S$}.
\end{definition}

As an example, consider the distance monoid $\cS_3$, and let $\cF$ be the class of metric triangles of odd perimeter bounded by $5$. In Figure \ref{pathsub}, we list the elements of $\cF$, and the associated path substitutions over $\{1,2\}$.

\begin{figure}[H]
\begin{tabular}{|>{\centering}m{2.3cm} | >{\centering}m{8.7cm} |}
\hline
 {\small elements of $\cF$} & {\small path substitutions over $\{1,2\}$} \tabularnewline
 \hline
 {\footnotesize
\begin{tikzpicture} 
\draw[fill=black] (0,0) circle (2pt);
\draw[fill=black] (1.5,0) circle (2pt);
\draw[fill=black] (0.75,1.3) circle (2pt);
\node at (0.75,-0.2) {$1$};
\node at (0.2,0.8) {$1$};
\node at (1.3,0.8) {$1$};
\node at (0.75,1.5) {$~$};
\draw[thick] (0,0) -- (1.5,0) -- (0.75,1.3) -- (0,0);
\end{tikzpicture}
}
&
 {\footnotesize
\begin{tikzpicture} 
\draw[fill=black] (0,0) circle (2pt);
\draw[fill=black] (1.5,0) circle (2pt);
\draw[fill=black] (0.75,1.3) circle (2pt);
\node at (0.75,-0.2) {$1$};
\node at (0.2,0.8) {$1$};
\node at (1.3,0.8) {$1$};
\node at (0.75,1.5) {$~$};
\draw[thick] (0,0) -- (1.5,0) -- (0.75,1.3) -- (0,0);
\end{tikzpicture}
}
 \tabularnewline 
\hline
 {\footnotesize
\begin{tikzpicture} 
\draw[fill=black] (0,0) circle (2pt);
\draw[fill=black] (1.5,0) circle (2pt);
\draw[fill=black] (0.75,1.3) circle (2pt);
\node at (0.75,-0.2) {$1$};
\node at (0.2,0.8) {$2$};
\node at (1.3,0.8) {$2$};
\node at (0.75,1.5) {$~$};
\draw[thick] (0,0) -- (1.5,0) -- (0.75,1.3) -- (0,0);
\end{tikzpicture}
}
&
{\footnotesize 
\begin{tikzpicture} 
\draw[fill=black] (0,0) circle (2pt);
\draw[fill=black] (1.5,0) circle (2pt);
\draw[fill=black] (0.75,1.3) circle (2pt);
\node at (0.75,-0.2) {$1$};
\node at (0.2,0.8) {$2$};
\node at (1.3,0.8) {$2$};
\node at (0.75,1.5) {$~$};
\draw[thick] (0,0) -- (1.5,0) -- (0.75,1.3) -- (0,0);
\end{tikzpicture}
\hspace{1pt}
\begin{tikzpicture} 
\draw[fill=black] (0,0) circle (2pt);
\draw[fill=black] (1.5,0) circle (2pt);
\draw[fill=black] (0.75,1.3) circle (2pt);
\draw[fill=black] (0.2,0.75) circle (2pt);
\node at (0.75,-0.2) {$1$};
\node at (-0.05,0.5) {$1$};
\node at (0.35,1.2) {$1$};
\node at (1.3,0.8) {$2$};
\node at (0.75,1.5) {$~$};
\draw[thick] (0,0) -- (1.5,0) -- (0.75,1.3) -- (0.2,0.75) -- (0,0);
\draw[dashed] (0,0) -- (0.75,1.3);
\end{tikzpicture}
\hspace{1pt}
\begin{tikzpicture} 
\draw[fill=black] (0,0) circle (2pt);
\draw[fill=black] (1.5,0) circle (2pt);
\draw[fill=black] (0.75,1.3) circle (2pt);
\draw[fill=black] (0.2,0.75) circle (2pt);
\draw[fill=black] (1.3,0.75) circle (2pt);
\node at (0.75,-0.2) {$1$};
\node at (-0.05,0.5) {$1$};
\node at (0.35,1.2) {$1$};
\node at (1.55,0.5) {$1$};
\node at (1.15,1.2) {$1$};
\node at (0.75,1.5) {$~$};
\draw[thick] (0,0) -- (1.5,0) -- (1.3,0.75) -- (0.75,1.3) -- (0.2,0.75) -- (0,0);
\draw[dashed] (0,0) -- (0.75,1.3) -- (1.5,0);
\end{tikzpicture}
\hspace{1pt}
\begin{tikzpicture} 
\draw[fill=black] (0,0) circle (2pt);
\draw[fill=black] (1.5,0) circle (2pt);
\draw[fill=black] (0.75,1.3) circle (2pt);
\draw[fill=black] (0.75,0.4) circle (2pt);
\node at (0.75,-0.2) {$1$};
\node at (0.4,0.45) {$1$};
\node at (1.1,0.45) {$1$};
\node at (.9,0.8) {$1$};
\node at (0.75,1.5) {$~$};
\draw[thick] (0,0) -- (1.5,0) -- (0.75,0.4) -- (0.75,1.3);
\draw[thick] (0,0) -- (0.75,0.4);
\draw[dashed] (0,0) -- (0.75,1.3);
\draw[dashed] (0.75,1.3) -- (1.5,0);
\end{tikzpicture}
\hspace{2pt}
}  \tabularnewline 
\hline
  \end{tabular}
  \caption{\small Path substitutions are given up to isomorphism, and dotted lines represent original distances that have been replaced with paths. Altogether, the five pairwise non-isomorphic structures in the table make up the class $\cF^*_{\{1,2\}}$. Note that any metric assignment of distances from $\cS_3$ to pairs of points connected by dotted lines leads to a triangle in $\cF$. Therefore any $\cF$-free $\cS_3$-metric space is also $\cF^*_{\{1,2\}}$-free, i.e., $\cF$ is closed under path substitutions over $\{1,2\}$. Lemma \ref{lem:cycles} below generalizes this to arbitrary odd $n\geq 3$ (in place of $5$).}
  \label{pathsub}
\end{figure}

We now prove the main theorem of this section, the statement of which mirrors the Herwig-Lascar Theorem. The proof strategy follows Solecki's proof of the Hrushovski property for $\cK_{\R^{\geq0}}$ in \cite{Sol}.

\begin{theorem}\label{thm:O}
Fix an archimedean distance monoid $\cR$. Suppose $\cF$ is a finite class of finite $\cR$-dominated $\cR$-metric spaces, and $\cA=(A,d_A)$ is a finite $\cF$-free $\cR$-metric space. Suppose $S\seq R^{>0}$ is finite, $\cF$ is closed under path substitutions over $S$, and $S$ contains any nonzero distance occuring in $\cA$ or in an element of $\cF$. If $\cA$ has the extension property in the class of $\cF$-free $\cR$-metric spaces, then $\cA$ has the extension property in the class of finite $\cF$-free $\cR$-metric spaces. 
\end{theorem}
\begin{proof}
The first main claim is that $\cE^S_\cR$ is finite. Indeed, if $s=\min S$ then, since $\cR$ is archimedean, there is some $k>0$ such that $r\leq ks$ for any $r\in S$. It follows that if $(r_0,\ldots,r_n)\in \Sigma^S_\cR$ then $n<k$, and so $|\cE^S_\cR|=|\Sigma^S_\cR|\leq |S|^k$. 

Next, by Lemma \ref{lem:ext}, $\cF^*_S$ is a finite class of finite $\cR$-semimetric spaces.  Note that any $\cR$-metric space is a path substitution of itself, and so $\cF\seq\cF^*_S$ by assumption on $S$. Altogether, if $\cF^*=\cE^S_\cR\cup\cF^*_S$, then $\cF^*$ is a finite class of finite $\cL_S$-structures, and any $\cF$-free $\cR$-metric space is $\cF^*$-free.

Suppose $\cA$ has the extension property in the class of \emph{all} $\cF$-free $\cR$-metric spaces, and so there is an $\cF$-free $\cR$-metric space $\cU$ (possibly infinite) such that $\cA$ is a subspace of $\cU$ and any partial isometry of $\cA$ extends to a total isometry of $\cU$. We want to replace $\cU$ with a \emph{finite} $\cF$-free $\cR$-metric space, which similarly extends isometries of $\cA$. By assumption on $S$, any partial $\cL_S$-automorphism of $\cA$ extends to a total $\cL_S$-automorphism of $\cU$. By the Herwig-Lascar Theorem, there is a \emph{finite} $\cF^*$-free $\cL_S$-structure $\cC$ such that $\cA$ is a substructure of $\cC$ and any partial $\cL_S$-automorphism of $\cA$ extends to a total $\cL_S$-automorphism of $\cC$.

Let $C$ be the underlying universe of $\cC$, and define a graph relation $E$ on $C$ such that, given distinct $x,y\in C$, $E(x,y)$ holds if and only if $\cC\models d(x,y)=r$ for some $r\in S$. Note that $(A,E)$ is a complete subgraph of $(C,E)$. Let $B\seq C$ be the connected component of $C$ containing $A$. Define $d_B\colon B\times B\to R$ such that $d_B(x,x)=0$ for all $x\in B$ and, given distinct $x,y\in B$,  $d_B(x,y)$ is the smallest element of $\cR$ obtained as a sum of distances on an $E$-path from $x$ to $y$. It is easy to see that $d_B$ is an $\cR$-metric on $B$.  Moreover, $d_B|_{A\times A}=d_A$, since $\cC$ is $\cE^S_\cR$-free and $S$ contains all nonzero distances in $\cA$. Finally, since $B$ is $E$-connected, it is straightforward to show that any partial isometry of $\cA$ extends to a total isometry of $\cB:=(B,d_B)$ (using the defining property of $\cC$). Further detail can be found in \cite{Sol}.

Note that $\cB$ is finite, since $B\seq C$. It remains to show that $\cB$ is $\cF$-free. Since $\cC$ is $\cE^S_\cR$-free as an $\cL_S$-structure, it follows that for any distinct $a,b\in B$ there is at most one $r\in S$ such that $\cC\models d(a,b)=r$ and, in this case, we have  $d_B(a,b)=r$. Suppose, toward a contradiction, that there is $Y\seq B$, with $(Y,d_B)\in\cF$. Then, for any $x,y\in Y$, there is a sequence $\pi_{x,y}=(x_0,\ldots,x_n)$ of elements of $B$ such that:
\begin{enumerate}[$(i)$]
\item $x_0=x$ and $x_n=y$;
\item for all $0\leq i<n$, $d_B(x_i,x_{i+1})\in S$ and $\cC\models d(x_i,x_{i+1})=d_B(x_i,x_{i+1})$;
\item $d_B(x_0,x_1)\p d_B(x_1,x_2)\p\ldots\p d_B(x_{n-1},x_n)=d_B(x,y)$.
\end{enumerate}
Let $X=Y\cup\bigcup_{(x,y)\in Y\times Y}\pi_{x,y}$, and note that $X\seq B$. Define a partial function $\delta\colon X\times X\to R$ such that
\[
\dom(\delta)=\{(x_i,x_{i+1}):0\leq i<n,~(x_0,\ldots,x_n)=\pi_{x,y},~(x,y)\in Y\times Y\},
\]
and, given $(x,y)\in \dom(\delta)$, $\delta(x,y)=d_B(x,y)$. By construction, $(X,\delta)$ is a path substitution of $Y$ over $S$, which contradicts that $\cC$ is $\cF^*_S$-free.  
\end{proof}

\begin{remark}\label{rem:arch}
If $\cF=\emptyset$ then $\cF$ satisfies the assumptions of Theorem \ref{thm:O}. Combined with Proposition \ref{prop:ctble}, as well as the existence of an $\cR$-Urysohn space for any countable $\cR$, this yields Theorem \ref{thm:arch}. 
\end{remark}

We now apply Theorem \ref{thm:O} to the examples discussed in Section \ref{sec:intro}. As a warm up, fix $n\geq 3$ and note that, as an $\cS_2$-metric space, the complete graph $K_n$ is $\cS_2$-dominated and has no proper path substitutions (over $\{1,2\}$). Moreover, any finite $K_n$-free graph has the extension property in the class of $K_n$-free graphs (witnessed by the Henson graph $\cH_n$). Therefore, the class of finite $K_n$-free graphs has the Hrushovski property. Of course, the direct proof of this from the Herwig-Lascar theorem does not require the translation  to $\cS_2$-metric spaces. Thus, for the framework developed here, the interesting examples are metric spaces omitting triangles of odd perimeter.

\begin{lemma}\label{lem:cycles}
Fix $n\geq 3$ odd, and let $n_*=\frac{n+1}{2}$. Let $\cF^c_n$ be the class of metric triangles of odd perimeter bounded by $n$. Let $S=\{1,\ldots,n_*\}$. Then $\cF^c_n$ is $\cS_{n_*}$-closed under path substitutions over $S$.
\end{lemma}
\begin{proof}
Suppose $(X,\delta)$ is a path substitution over $S$ of some triangle in $\cF^c_n$. If there is a weak homomorphism from $(X,\delta)$ to an $\cS_{n_*}$-metric space $(A,d)$, then there is a (possibly non-simple) cycle $(x_1,x_2\ldots,x_m,x_1)$ in $A$, with $m\geq 3$, such that $d(x_1,x_2)+\ldots+d(x_{m-1},x_m)+d(x_m,x_1)$ is odd and bounded by $n$. Therefore, it suffices to fix an $\cF^c_n$-free $\cS_{n_*}$-metric space $(A,d)$ and a sequence $(x_1,\ldots,x_m)$ of elements of $A$, with $m\geq 3$, and show that, if $p:=d(x_1,x_2)+\ldots+d(x_{m-1},x_m)+d(x_m,x_1)$ and $p\leq n$, then $p$ is even.

We proceed by induction on $m$. If $m=3$ and $(x_1,x_2,x_3)$ is injective then the claim follows since $A$ is $\cF^c_n$-free. Otherwise, after a cyclic rotation, we may assume $x_1=x_2$, and so $p=2d(x_2,x_3)$, which is even. Assume the result for $m$ and fix a sequence $(x_1,\ldots,x_{m+1})$ of elements of $A$. Define
\[
q = d(x_1,x_2)+\ldots+d(x_{m-1},x_m)\makebox[.5in]{and} r = d(x_m,x_{m+1})+d(x_{m+1},x_1).
\]
Set $p=q+r$, and suppose $p\leq n$. We want to show that $p$ is even. Since $d(x_m,x_1)\leq r$, we have $q+d(x_m,x_1)\leq q+r=p\leq n$. By induction, $q+d(x_m,x_1)$ is even, and so $q$ and $d(x_m,x_1)$ have the same parity. We also have $d(x_m,x_1)\leq q$, and so $d(x_1,x_m)+r\leq q+r=p\leq n$. By the base case, $d(x_1,x_m)+r$ is even, and so $d(x_m,x_1)$ and $r$ have the same parity. Altogether, $q$ and $r$ have the same parity, and so $p=q+r$ is even.
\end{proof}

We can now prove our desired results. Recall that, given $n\geq 3$ odd, $\Gamma^c_n$ denotes the countable, universal and existentially closed graph omitting cycles of odd length bounded by $n$. Equipped with the path metric, $\Gamma^c_n$ is the \Fraisse\ limit of the class of finite $\cF^c_n$-free $\cS_{n_*}$-metric spaces.

\begin{proof}[Proof of Theorem \ref{thm:oddT}]
Part $(a)$. We apply Theorem \ref{thm:O}, with $\cR=\cS_{n_*}$, $\cF=\cF^c_n$, and $S=\{1,\ldots,n_*\}$. Clearly, $\cS_{n_*}$ is archimedean and $\cF^c_n$ is a finite class of finite $\cS_{n_*}$-metric spaces. By Lemma \ref{lem:cycles}, $\cF^c_n$ is closed under path substitutions over $S$. Moreover, $\Gamma^c_n$ witnesses that any finite $\cF^c_n$-free $\cS_{n_*}$-metric space has the extension property in the class of $\cF^c_n$-free $\cS_{n_*}$-metric spaces. Altogether, we only need to show that each triangle in $\cF^c_n$ is $\cS_{n_*}$-dominated. For, this it suffices to show that $n_*$ does not appear as a distance in any triangle in $\cF^c_n$. Indeed, if $\{a_1,a_2,a_3\}$ is a $3$-point $\cS_{n_*}$-metric space and $d(a_1,a_2)=n_*$, then $d(a_2,a_3)+d(a_3,a_1)\geq n_*$, and so $d(a_1,a_2)+d(a_2,a_3)+d(a_3,a_1)\geq 2n_*=n+1$. 

Part $(b)$. Let $d$ denote the path metric on $\Gamma^c_n$. Then $\Aut(\Gamma^c_n)=\Isom(\Gamma^c_n,d)$, and so we want to show $\Isom(\Gamma^c_n,d)$ has ample generics. Once again, we use the characterization in \cite{KeRo} described before Fact \ref{fact:ample}. Since the standard amalgamation of $\cS_{n_*}$-metric spaces preserves the property of being $\cF^c_n$-free, the Hrushovski property for $\cK^c_n$ implies weak amalgamation for $(\cK^c_n)^{p,m}$, for all $m>0$, using the same argument as for $\cK_{\cS_{n^*}}$. The joint embedding property for $(\cK^c_n)^{p,m}$ is similarly straightforward. 
\end{proof}

\subsection*{Acknowledgements}
I would like to thank the referee for many helpful suggestions, which significantly improved the final draft.

\end{document}